\documentclass[11pt,reqno]{amsart}

\usepackage{graphicx}

\usepackage{appendix}

\usepackage{mleftright}

\usepackage{oplotsymbl}

\usepackage{wasysym}


\DeclareMathAlphabet{\mathpzc}{OT1}{pzc}{m}{it}
\DeclareMathAlphabet{\mathantt}{OT1}{antt}{li}{it}

\usepackage{comment}
\usepackage{enumerate}
\usepackage{esvect}
\usepackage{bm}
\usepackage{amsmath}
\usepackage{amssymb}
\usepackage{mathtools}

\usepackage{mathrsfs}
\usepackage{color}

\usepackage[normalem]{ulem}

\usepackage{url}

\newtheorem{theorem}{Theorem}[section]

\newtheorem{cor}[theorem]{Corollary}
\newtheorem{lemma}[theorem]{Lemma}

\newtheorem{ex}[theorem]{Example}

\usepackage{bbm}

\DeclareSymbolFont{bbold}{U}{bbold}{m}{n}
\DeclareSymbolFontAlphabet{\mathbbold}{bbold}
\DeclareMathOperator{\sgn}{sgn}

\usepackage{stmaryrd}

\usepackage{hyperref}
\hypersetup{hidelinks,colorlinks}

\hypersetup{
    linkcolor=red,
    citecolor=blue,
    urlcolor=black,
 }

\usepackage{dsfont}

\usepackage{mleftright} 

\title{On Digit-Appending Walks Along Linear Recurrences}

\author[A. Mosahab]{Arya Mosahab}
\address{Department of Mathematics, University of Manchester}
\email{arya.mosahab@student.manchester.ac.uk}

\author[S. Boyyakar]{Sahitya Boyyakar}
\address{Malla Reddy Deemed To Be University}
\email{\href{mailto:23j41a05d4@mallareddyuniversity.ac.in}{\texttt{23j41a05d4@mallareddyuniversity.ac.in}}}

\author[G. Bruda]{Glenn Bruda}
\address{School of Mathematics, Georgia Institute of Technology}
\email{gbruda3@gatech.edu}

\author[S. D. Kominers]{Scott Duke Kominers}
\address{Harvard Business School; Department of Economics and Center of Mathematical Sciences and Applications, Harvard University; and a16z crypto}
\email{kominers@fas.harvard.edu}

\author[S. J. Miller]{Steven J. Miller}
\address{Department of Mathematics, Williams College}
\email{sjm1@williams.edu}

\allowdisplaybreaks

\begin{document}

\begin{abstract}
We study digit-appending walks in base $b\geq 2$ along integer linear recurrences with minimal characteristic polynomials having a unique simple dominant root $\alpha$ that lies outside the unit circle. We prove that if $\alpha$ and $b$ are multiplicatively independent, then no infinite walk can have uniformly bounded appended digit lengths. In contrast, we show that when the appended digit length can be unbounded, the same independence hypothesis implies the existence of infinite walks.  
\end{abstract}

\maketitle

\thispagestyle{empty}

\section{Introduction}
The idea of \emph{walking to infinity} arises naturally in number theory. Given a set of integers or other discrete points, one may ask whether it is possible to move indefinitely through the set while placing a uniform restriction on the allowed steps. For the ordinary primes, one cannot walk to infinity using steps of uniformly bounded size, since there are arbitrarily large gaps between consecutive primes. 

An analogous question for the Gaussian primes leads to the Gaussian moat problem \cite{GethnerWagonWick1998}, which asks whether one can travel arbitrarily far from the origin through Gaussian primes while keeping the Euclidean length of each step uniformly bounded. These and related questions motivate the study of other ways in which the steps of a walk may be restricted.

A simple operation on an integer is to append digits in some base $b\geq 2$. If the starting integer is in a given sequence, then it is natural to ask if the resulting integer is in the sequence, and if it is, whether this process can be continued to give us an infinite walk along sequence terms. 

Questions of this type are related to earlier work on concatenations of terms in a sequence, by Banks and Luca \cite{BanksLuca2005}. They considered binary linear recurrences in an arbitrary base $b\geq2$ under some mild assumptions, and proved that there are finitely many terms in such a sequence that can be concatenated to give further terms. 

However, in our case, we have the freedom to choose the digits we want to append and are not restricted to only concatenating sequence terms. This notion of \textit{digit-appending walks} was used by Miller et al.\@ \cite{MillerEtAl2024}, who studied digit-appending walks on number theory sequences, including the primes, square-free numbers, and perfect squares.\footnote{The results on the square-free case were refined and optimized by Chen et al.~\cite{Chenetal2026}, who then later learned that line of inquiry had recovered seemingly-forgotten work of Mirsky~\cite{Mirsky1947}.}
For the Fibonacci sequence, Miller, Peng, Popescu, and Wattanawanichkul \cite{MillerEtAl2022} studied walks in base $10$ when the number of digits appended at each step is bounded. 

Kominers \cite{Kominers2025} subsequently studied the Fibonacci problem in an arbitrary base. For every base $b\geq2$, he proved uniform termination bounds for the walks when at most $N$ digits are appended at each step, and extended the method to second-order Lucas sequences. He found that a bound on the number of appended digits eventually forces a bound on the corresponding increase in the sequence index, and used exact identities of these sequences to bound walk lengths. 

We consider an integer sequence $(a_n)_{n\geq0}$ satisfying a constant-depth linear recurrence with constant rational coefficients, whose minimal characteristic polynomial has a unique simple dominant root $\alpha$ lying outside the unit circle. Fix an integer base $b\geq2$. We say that there is a \emph{digit-appending step} from $a_m$ to $a_{m+k}$ if
\begin{equation}\label{eq:digit-appending}
a_{m+k}=b^ta_m+d\sgn(a_m),
\qquad
k,t\in\mathbb{Z}_{\geq1},
\qquad
d\in\mathbb{Z},\quad 0\leq d<b^t.
\end{equation}

The following is a simple example of a sequence for which an infinite digit-appending walk exists (see Example \ref{nonexample} for a nonexample). 

\begin{ex}
Consider the sequence $a_n = b^n-1$, which satisfies the linear recurrence $a_{n+2} = (b+1)a_{n+1}-ba_n$,with $b$ as its unique simple dominant root. Also note that it can be written as $a_{n+1} = ba_n+(b-1)$, so every step in the recurrence is a digit-appending step (satisfies \eqref{eq:digit-appending} with $t=1$ and $d=b-1$). Therefore, an infinite digit-appending walk clearly exists.
\end{ex}

Our approach replaces the exact Fibonacci and Lucas identities by dominant root asymptotics. For bounded digit lengths, this gives analogues of the index-increment bound and rigidity established in \cite[Lemmas 3.2--3.3]{Kominers2025}. If $t\leq N$, then sufficiently far along the recurrence, the index increment $k$ is also bounded. Hence only finitely many pairs $(k,t)$ can occur, and the asymptotic estimate of the sequence forces every sufficiently late digit-appending step to satisfy
\[
\alpha^k = b^t.
\]
Consequently, if $\alpha$ and $b$ are \emph{multiplicatively independent} (that is, if $\alpha^k\neq b^t$ for all positive integers $k$ and $t$), then no infinite bounded digit--appending walk exists. Thus, the bounded case extends the rigidity from Fibonacci and Lucas sequences to general integer linear recurrences with constant coefficients, under our dominant root assumption.

We also consider the case in which the appended digit length is unrestricted. In this case, we show that if $|\alpha|$ and $b$ are multiplicatively independent, then every sufficiently large term has infinitely many outgoing digit-appending steps, and hence that an infinite walk exists.

The following are our main results.

\begin{theorem} \label{inf walk condition}
For each $N\geq 1$, there exists an $M$ such that every digit-appending step \eqref{eq:digit-appending} with $m\geq M$ and $1\leq t\leq N$ satisfies
\[
\alpha^k = b^t.
\]
\end{theorem}

\begin{theorem} \label{unbounded thm}
Suppose that $r$ and $b$ are multiplicatively independent and that the appended digit length is unrestricted. Then there exists $M$ such that for every $m\geq M$, $a_m$ has infinitely many outgoing digit-appending steps. Consequently, an infinite digit-appending walk exists.
\end{theorem}

\section{Preliminaries}
Let $(a_n)_{n\geq0}$ be an integer sequence satisfying a constant-depth linear recurrence with constant rational coefficients, and let $\chi$ be its minimal characteristic polynomial over $\mathbb{Q}$. Let $\lambda_1,\dots,\lambda_s$ be the distinct nonzero roots of $\chi$, with multiplicities $\mu_1,\dots,\mu_s$, respectively. For all sufficiently large $n$, by the generalized Binet theorem \cite[Theorem IV.9]{FlajoletSedgewick2009}, we have
\begin{equation}\label{eq:exp-poly}
a_n=\sum_{i=1}^s P_i(n)\lambda_i^n,
\end{equation}
where each $P_i$ is a polynomial of degree at most $\mu_i-1$.\footnote{A possible zero root of $\chi$ contributes only to finitely many initial terms, which is why we state this expansion for sufficiently large $n$.}

We assume throughout that $\chi$ has a unique dominant root $\alpha$, that $\alpha$ is simple, and that $r:=|\alpha|>1$.
Since $\chi$ has rational coefficients, uniqueness of the dominant root implies that $\alpha$ is real. Let $A$ be the coefficient of $\alpha^n$ in \eqref{eq:exp-poly}. By minimality of $\chi$, we have $A\neq0$. Thus, for all sufficiently large $n$,
\[
a_n=A\alpha^n+
\sum_{\substack{1\leq i\leq s\\\lambda_i\neq\alpha}}
P_i(n)\lambda_i^n.
\]
In addition, we define
\[
p:=\max\bigl(\{|\lambda_i|:\lambda_i\neq\alpha\}\cup\{0\}\bigr),
\]
with the convention that $p = 0$ if there are no other characteristic roots.

\begin{lemma} \label{elementary}
For every $q\in(p/r,1)$ there is a constant $C>0$ such that
\[
a_n = A\alpha^n+E_n,
\qquad\text{with $|E_n|\leq C(qr)^n$}
\]
for all $n\geq0$.
\end{lemma}

\begin{proof}
Let $\lambda_i\neq\alpha$ be a non-dominant root of the characteristic polynomial. Since $p/r<q<1$ and $|\lambda_i|/r\leq p/r$, we have $|\lambda_i|/(qr)<1$ and therefore $n^{\deg(P_i(n))}(|\lambda_i|/(qr))^n$ is bounded. Hence,
\[
P_i(n)\lambda_i^n = O\left(n^{\deg(P_i(n))}|\lambda_i|^n\right) = O\left(n^{\deg(P_i(n))}\left(\frac{|\lambda_i|}{qr}\right)^n(qr)^n\right) = O\left((qr)^n\right);
\]
as there are only finitely many non-dominant roots, summing these estimates gives
\[
E_n = \sum_{\lambda_i\neq\alpha}P_i(n)\lambda_i^n = O\left((qr)^n\right).\qedhere
\] 
\end{proof}

The elementary observation in Lemma~\ref{elementary} allows us to prove that for sufficiently large $m$, $\alpha^ka_m$ is a good estimate for $a_{m+k}$. Since the error only depends on the index $m$, the approximation remains valid uniformly even when $k$ varies with $m$.

\begin{lemma} \label{ratio estimate}
Fix $q\in(p/r,1)$. There exist constants $M\geq0$ and $C>0$ such that $a_m\neq0$ and
\begin{equation} \label{ratio 1}
\left|\frac{a_{m+k}}{\alpha^ka_m}-1\right| \leq Cq^m  
\end{equation}
for every $m\geq M$ and $k\geq 0$.
Consequently, for all $\eta\in(0,1)$ there is $M_\eta$ such that
\begin{equation}\label{ratio 2}
(1-\eta)r^k|a_m|\leq |a_{m+k}|\leq (1+\eta)r^k|a_m|  
\end{equation}
for every $m\geq M_\eta$ and every $k\geq0$.
\end{lemma}

\begin{proof}
By Lemma \ref{elementary}, we may write $a_n=A\alpha^n(1+\varepsilon_n)$ such that $|\varepsilon_n|\leq C_1q^n$ for some $C_1>0$. We pick an $M$ sufficiently large so that $|\varepsilon_n|\leq 1/2$ for all $n\geq M$. Then, for $m\geq M$,
\begin{align*}
\left|\frac{a_{m+k}}{\alpha^ka_m}-1\right| = \left|\frac{1+\varepsilon_{m+k}}{1+\varepsilon_m}-1\right| = \left|\frac{\varepsilon_{m+k}-\varepsilon_m}{1+\varepsilon_m}\right| \leq \frac{|\varepsilon_{m+k}|+|\varepsilon_m|}{\left|1-|\varepsilon_m|\right|} \leq 2\left(|\varepsilon_{m+k}|+|\varepsilon_m|\right) \leq 4C_1q^m.
\end{align*}
We can let $C = 4C_1$ and have therefore proven \eqref{ratio 1}.

To prove \eqref{ratio 2}, we pick $M_\eta$ sufficiently large so that for all $m\geq M_\eta$, we have $Cq^m\leq \eta$. Then, by the reverse triangle inequality,
\[
\left|\frac{|a_{m+k}|}{r^k|a_m|}-1\right|\leq\left|\frac{a_{m+k}}{\alpha^ka_m}-1\right|\leq\eta,
\]
and therefore
\[
(1-\eta)r^k|a_m|\leq |a_{m+k}|\leq (1+\eta)r^k|a_m|.\qedhere
\]
\end{proof}

\section{Bounded appended digit lengths}
If the digit length in \eqref{eq:digit-appending} is bounded, that is, $t\leq N$ for some $N\geq 1$, we may use Lemma \ref{ratio estimate} to prove that the step size $k$ is also bounded. This yields the following analogue of \cite[Lemma 3.2]{Kominers2025}.

\begin{lemma} \label{k bound}
For every $N\geq1$, there exists an $M$ such that for all $m\geq M$, every digit-appending step \eqref{eq:digit-appending} with $1\leq t\leq N$ satisfies $r^k < 4b^N$. In particular, $k < \log_r4b^N$.
\end{lemma}

\begin{proof}
First, observe $|a_{m+k}| = |b^ta_m+d\sgn(a_m)|  = |(b^t|a_m|+d)\sgn(a_m) | = b^t|a_m|+d$. Choose $M$ sufficiently large so that \eqref{ratio 2} is satisfied with $\eta = 1/2$ and $a_m\neq 0$ for every $m\geq M$. Then, we have
\[
|a_{m+k}| = b^t|a_m|+d<b^N|a_m|+b^N\leq2b^N|a_m|
\]
and also
\[
|a_{m+k}|\geq \frac{1}{2}r^k|a_m|.
\]
Combining the two preceding inequalities gives $r^k<4b^N$; the bound on $k$ follows immediately.
\end{proof}

Lemma~\ref{k bound} reduces the problem of walking along the sequence via bounded digit--appending steps to finitely many pairs $(k,t)$, allowing us to compare digit-appending \eqref{eq:digit-appending} with the asymptotic approximation $a_{m+k}\sim \alpha^ka_m$. We now prove Theorem \ref{inf walk condition}.

\begin{proof}[Proof of Theorem \ref{inf walk condition}]
Let 
\[
S_N := \{|\alpha^k-b^t|:1\leq k\leq \left\lceil \log_r4b^N \right \rceil, 1\leq t\leq N, \alpha^k\neq b^t\}.
\]
By Lemma \ref{k bound}, there exists $M_0$ so that every bounded digit--appending step starting at $m\geq M_0$ satisfies
\[
1\leq k\leq \left\lceil \log_r4b^N \right \rceil.
\]
If $S_N$ is empty, then $\alpha^k = b^t$ for all $1\leq k\leq \left\lceil \log_r4b^N \right \rceil$ and $1\leq t\leq N$. Hence, every bounded digit--appending step starting at $m\geq M_0$ satisfies $\alpha^k = b^t$. Suppose now that $S_N$ is non-empty. Since $k$ and $t$ are bounded integers,  $S_N$ is finite and (thus) has a minimum. 

Fix $q\in (p/r,1).$
By Lemma \ref{ratio estimate}, there exists $M_1\geq M_0$ and $C>0$ such that for all $m\geq M_1$ and $k\geq0$,  
\[
\left|\frac{a_{m+k}}{a_m}-\alpha^k\right|=|\alpha^k|\left|\frac{a_{m+k}}{\alpha^ka_m}-1\right|\leq Cr^kq^m.
\]
Since we picked $M_1\geq M_0$, every digit-appending step starting at $m\geq M_1$ satisfies $1\leq k\leq \left\lceil \log_r4b^N \right \rceil$, thus
\[
\left|\frac{a_{m+k}}{a_m}-\alpha^k\right|\leq Cr^{\left\lceil \log_r4b^N \right \rceil}q^m.
\]
Since $0<q<1$, there exists some $M_2\geq M_1$ such that for all $m\geq M_2$, 
\[
Cr^{\left\lceil \log_r4b^N \right \rceil}q^m< \frac{\min S_N}{2}.
\]
In addition, because of Lemma \ref{elementary}, there must exist an $M_3\geq M_2$ such that for all $m\geq M_3$, 
\[
\frac{b^N}{|a_m|}<\frac{\min S_N}{2}.
\]

Now suppose a bounded digit--appending step starts at some $m\geq M_3$. We have
\[
\frac{a_{m+k}}{a_m} = b^t+d\frac{\sgn(a_m)}{a_m} = b^t+\frac{d}{|a_m|}.
\]
Thus,
\[
\frac{a_{m+k}}{a_m}-b^t = \left|\frac{a_{m+k}}{a_m}-b^t\right| = \frac{d}{|a_m|}<\frac{b^t}{|a_m|}\leq \frac{b^N}{|a_m|}<\frac{\min S_N}{2}.
\]
Then, by the triangle inequality,
\[
|\alpha^k-b^t|\leq \left|\alpha^k-\frac{a_{m+k}}{a_m}\right|+\left|\frac{a_{m+k}}{a_m}-b^t\right|<\frac{\min S_N}{2}+\frac{\min S_N}{2} = \min S_N
\]
However, if $\alpha^k\neq b^t$, we have $|\alpha^k-b^t|\in S_N$, and therefore
\[
|\alpha^k-b^t|\geq \min S_N,
\]
giving a contradiction.
\end{proof}

Using Theorem~\ref{inf walk condition}, we immediately obtain a sufficient condition for when infinite walks are impossible. 

\begin{cor} \label{suff cond}
If $\alpha$ and $b$ are multiplicatively independent, then there is no infinite digit-appending walk with uniformly bounded appended-digit lengths.
 \end{cor}

\begin{proof}
The corollary follows immediately from the contrapositive of Theorem \ref{inf walk condition}.
\end{proof}

However, the converse is not true. We now construct a simple example where $\alpha$ and $b$ are multiplicatively dependent, but an infinite digit-appending walk does not exist.

\begin{ex} \label{nonexample}
Let $b\geq2$ and consider the sequence $a_n = b^n+1$, which satisfies the linear recurrence $a_{n+2}=(b+1)a_{n+1}-ba_n$: an infinite digit-appending walk along this sequence does not exist in base $b$. 
\end{ex}

\begin{proof}
Suppose there exists a digit-appending step 
\[
a_{m+k} = b^ta_m+d \qquad 0\leq d<b^t
\]
with $m\geq1$. Rearranging gives
\[
b^{m+k}+1 = b^t(b^m+1)+d = b^{m+t}+b^t+d,
\]
and therefore $d \equiv 1\pmod{b^t}$. Since $0\leq d<b^t$, we must then have $d = 1$ and thus $$b^{m+k} = b^{m+t}+b^t.$$ Considering both sides modulo $b^{m+t}$, the preceding identity gives $b^t \equiv 0 \pmod{b^{m+t}}$---a contradiction. Thus no digit-appending step from $a_m$ with $m\geq1$ exists, so an infinite digit-appending walk clearly cannot exist.
\end{proof}

\begin{cor}\label{alpha is integer}
If there exists an infinite digit-appending walk whose appended digit lengths are uniformly bounded, then the dominant root $\alpha$ must be an integer.
\end{cor}

\begin{proof}
Since the indices in an infinite digit-appending walk strictly increase, Theorem~\ref{inf walk condition} gives positive integers $k,t$
such that $\alpha^k=b^t$.

As noted above, as $\alpha$ is the unique dominant root of the recurrence's characteristic polynomial, it must be real. 
If $\alpha$ is irrational, then let $\alpha^*$ be one of its algebraic conjugates; since the characteristic polynomial has
rational coefficients, $\alpha^*$ is also one of its roots. Conjugating $\alpha^k=b^t$ gives $|\alpha^*|^k=b^t =|\alpha|^k $. Thus, $\alpha$ and $\alpha^*$ have the same modulus, contradicting the uniqueness of the dominant root. In addition, $\alpha$ cannot be of the form $e/f$ with $\gcd(e,f) = 1$, because then $\gcd(e^k,f^k) = 1$ and hence no power of $\alpha$ will be an integer. Therefore, we see that $\alpha$ must in fact be an integer.
\end{proof}

The characteristic polynomial of the Fibonacci sequence has dominant root $\varphi$, which is irrational. Therefore, Corollary~\ref{alpha is integer} rules out infinite digit-appending walks with uniformly bounded appended digit lengths in every base, consistent with Kominers's results \cite{Kominers2025} and those of Miller, Peng, Popescu, and Wattanawanichkul \cite{MillerEtAl2022}. 

\section{Unbounded appended digit lengths}

We now consider the case where $t$ is unbounded. With unbounded digit--appending walks, the argument in the previous section no longer applies because both $k$ and $t$ may vary through unbounded sets. In this case, as we show next, infinite walks are actually possible when $\alpha$ and $b$ are multiplicatively independent; indeed, multiplicative independence directly aids in constructing the walk.

We first prove a useful density argument.
\begin{lemma} \label{density}
If $r>1$ and $b\geq2$ are multiplicatively independent and $I\subset(0,\infty)$ is a non-empty open interval, there are infinitely many pairs of positive integers $(k,t)$ such that
\[
\frac{r^k}{b^t}\in I.
\]
In addition, the pairs may also be chosen with $k$ even.
\end{lemma}

\begin{proof}
We let $I = (u,v)$. Since $b>1$, it is sufficient to prove that we can find infinitely many pairs $(k,t)$ for which
\[
k\log_br-t\in (\log_bu,\log_bv).
\]
We pick an $x\in (\log_bu,\log_bv)\setminus\mathbb{Z}$ and fix a non-empty open interval
\[
(y,z)\subset (\log_bu,\log_bv) \cap (\lfloor x\rfloor,\lfloor x\rfloor+1).
\]
Therefore, $(y-\lfloor x\rfloor,z-\lfloor x\rfloor)\subset (0,1)$.

Since $\log_br\notin \mathbb{Q}$, by Kronecker's approximation theorem (see for example \cite[Chapter 1, Example 2.1]{KuipersNiederreiter1974}), for $K\geq1$, the set
\[
\{\{k\log_br\}:k\geq K\},
\]
where $\{k\log_br\}$ denotes the fractional part of $k\log_br$, is dense in $[0,1]$. Hence, there are arbitrarily large $k$ for which 
\[
\{k\log_br\}\in (y-\lfloor x\rfloor, z-\lfloor x\rfloor).
\]
For such $k$, we set 
\[
t = \lfloor k\log_br\rfloor-\lfloor x\rfloor.
\]
Thus, for sufficiently large $k$ we have $t\geq 1$ and
\[
k\log_br-t = \{k\log_br\}+\lfloor x\rfloor\in(y,z)\subset(\log_bu,\log_bv).
\]
Since $r^2$ and $b$ are also multiplicatively independent, the same argument applied to $r^2$ shows that the pairs can be chosen with $k$ even. 
\end{proof}

Using Lemma \ref{density}, we can now prove the existence of infinite walks in the case where $t$ is unbounded (Theorem \ref{unbounded thm}).

\begin{proof}[Proof of Theorem \ref{unbounded thm}]
We fix $q\in(p/r,1)$ and write $a_n = A\alpha^n+E_n$ with $|E_n|\leq C(qr)^n$ for some $C>0$. We have, for sufficiently large $M$, $a_n\neq0$ and $\sgn(a_n) = \sgn(A\alpha^n)$ for every $n\geq M$. We also fix $m\geq M$ and consider the open interval
\[
I = \left(\frac{|a_m|+\frac{1}{3}}{|A|r^m},\frac{|a_m|+\frac{2}{3}}{|A|r^m}\right).
\]

By Lemma \ref{density}, there exists positive integers $k_1<k_2<\cdots$ and corresponding positive integers $t_1,t_2,\dots$ such that 
\[
\frac{r^{k_i}}{b^{t_i}} \in I
\]
for every $i$ (if $\alpha<0$, we specifically choose even $k_i$). Equivalently,

\begin{equation}\label{unbounded ineq 1}
|a_m|+\frac{1}{3}<\frac{|A|r^{m+k_i}}{b^{t_i}}<|a_m|+\frac{2}{3}.
\end{equation}
In addition, by the reverse triangle inequality,
\[
\left|\frac{|a_{m+k_i}|}{b^{t_i}}-\frac{|A|r^{m+k_i}}{b^{t_i}}\right| \leq \frac{|E_{m+k_i}|}{b^{t_i}} \leq
C(qr)^{m+k_i}\frac{1}{b^{t_i}} = C(qr)^m q^{k_i} \frac{r^{k_i}}{b^{t_i}} < C(qr)^m q^{k_i}\frac{|a_m|+\frac{2}{3}}{|A|r^m}.
\]
Since $0<q<1$, we can pick $i$ sufficiently large so that $\left|\frac{|a_{m+k_i}|}{b^{t_i}}-\frac{|A|r^{m+k_i}}{b^{t_i}}\right|<\frac{1}{3}$. This, together with \eqref{unbounded ineq 1}, gives
\[
b^{t_i}|a_m|<|a_{m+k_i}|<b^{t_i}(|a_m|+1).
\]

Let 
\[
d_i = |a_{m+k_i}|-b^{t_i}|a_m|
\]
and note that $0<d_i<b^{t_i}$. If $\alpha>0$, $a_{m+k_i}$ and $a_{m}$ have the same sign. The same is true when $\alpha<0$, since we have picked even $k_i$ for every $i$. Therefore,
\[
a_{m+k_i} = b^{t_i}a_m+d_i\sgn(a_m)
\]
and we have found a digit-appending step starting at $a_m$. Since all sufficiently large $i$ work, there are in fact infinitely many such steps. For every $m\geq M$, the digit-appending step leads to a term with some index $m+k_i$, which is once again greater than $M$, so choosing steps inductively gives an infinite digit-appending walk. 
\end{proof}

The following is a simple example of the construction.

\begin{ex}
Let $b=10$ and consider the sequence $a_n=2^n$, which satisfies the recurrence $a_{n+1}=2a_n$. Since this recurrence has the unique simple dominant root $2$, and $2$ and $10$ are multiplicatively independent, Theorem \ref{unbounded thm} applies and infinitely many infinite digit-appending walks exist along this sequence.
\end{ex}

\section*{Acknowledgements}
This work was supported by NSF Grant DMS2341670. This work was carried out as part of the 2026 Polymath Jr.\@ program, and the authors thank its organizers for coordinating the program. 

This work was conducted while Kominers was visiting the Technological Innovation, Entrepreneurship, and Strategic Management (TIES) Group at the MIT Sloan School of Management; he greatly appreciates their hospitality.

\bibliography{main}{}

\providecommand{\bysame}{\leavevmode\hbox to3em{\hrulefill}\thinspace}
\providecommand{\MR}{\relax\ifhmode\unskip\space\fi MR }
\providecommand{\MRhref}[2]{%
  \href{http://www.ams.org/mathscinet-getitem?mr=#1}{#2}
}
\providecommand{\href}[2]{#2}
\begin{thebibliography}{1}

\bibitem{BanksLuca2005}
William~D. Banks and Florian Luca, \emph{Concatenations with binary recurrent
  sequences}, Journal of Integer Sequences \textbf{8} (2005), Article 05.1.3.

\bibitem{Chenetal2026}
Evan Chen, Chris Cummins, Ben Eltschig, Dejan Grubisic, Leopold Haller, Letong
  Hong, Andranik Kurghinyan, Kenny Lau, Hugh Leather, Seewoo Lee, Aram
  Markosyan, Ken Ono, Manooshree Patel, Gaurang Pendharkar, Vedant Rathi, Alex
  Schneidman, Volker Seeker, Shubho Sengupta, Ishan Sinha, Jimmy Xin, and
  Jujian Zhang, \emph{Dead ends in square-free digit walks},
  arXiv.org/2602.05095, 2026.

\bibitem{FlajoletSedgewick2009}
Philippe Flajolet and Robert Sedgewick, \emph{Analytic combinatorics},
  Cambridge University Press, 2009.

\bibitem{GethnerWagonWick1998}
Ellen Gethner, Stan Wagon, and Brian Wick, \emph{A stroll through the gaussian
  primes}, The American Mathematical Monthly \textbf{105} (1998), no.~4,
  327--337.

\bibitem{Kominers2025}
Scott~Duke Kominers, \emph{Uniform bounds for digit-appending {F}ibonacci
  walks}, The Fibonacci Quarterly (2026), forthcoming.

\bibitem{KuipersNiederreiter1974}
Lauwerens Kuipers and Harald Niederreiter, \emph{{Uniform Distribution of
  Sequences}}, John Wiley \& Sons, 1974.

\bibitem{MillerEtAl2024}
Steven~J. Miller, Fei Peng, Tudor Popescu, Joshua~M. Siktar, and Nawapan
  Wattanawanichkul, \emph{Walking to infinity along some number theory
  sequences}, INTEGERS \textbf{24} (2024).

\bibitem{MillerEtAl2022}
Steven~J. Miller, Fei Peng, Tudor Popescu, and Nawapan Wattanawanichkul,
  \emph{Walking to infinity on the {F}ibonacci sequence}, The Fibonacci
  Quarterly \textbf{60} (2022), no.~5, 293--298.

\bibitem{Mirsky1947}
L.~Mirsky, \emph{Note on an asymptotic formula connected with $r$-free
  integers}, The Quarterly Journal of Mathematics \textbf{os-18} (1947), no.~1,
  178--182.

\end{thebibliography}
\bibliographystyle{amsplain}

\end{document}